\documentclass[a4paper,11pt,notitlepage]{amsart}
\usepackage{amsmath,amssymb,mathrsfs,amsthm}

\usepackage{verbatim}
\usepackage[spanish,USenglish]{babel} 
\usepackage{color}
\usepackage{tikz}
\usepackage{graphicx}

\newtheorem{theorem}{Theorem}
\newtheorem{corollary}[theorem]{Corollary}

\newtheorem{conjecture}[theorem]{Conjecture}
\newtheorem{proposition}[theorem]{Proposition}

\newcommand{\N}{\mathbb{{N}}}

\begin{document}

\title[Sums of powers of Catalan triangle numbers]{Sums of powers of  Catalan triangle numbers}

\author[Miana]{Pedro J. Miana}
\address{Departamento de Matem\'aticas, Instituto Universitario de Matem\'aticas y Aplicaciones, Universidad de Zaragoza, 50009 Zaragoza, Spain.}
\email{pjmiana@unizar.es}

\author[Ohtsuka]{Hideyuki Ohtsuka}
\address{Bunkyo University High School, 1191-7, Kami, Ageo-city, Saitama
Pref., 362-0001, JAPAN.}
\email{otsukahideyuki@gmail.com}

\author[Romero]{Natalia Romero}
\address{Departamento de Matem\'aticas y Computaci\'{o}n, Universidad de La Rioja, 26004 Logro\~{n}o, Spain.}
\email{natalia.romero@unirioja.es}

\thanks{ P. J. Miana  has been partially supported  by Project MTM2013-42105-P, DGI-FEDER, of the MCYTS and Project E-64-FEDER, D.G. Arag\'on. N. Romero has been has been partially supported by  the Spanish Ministry of Economy and Competitiveness, Project MTM2014-52016-C2-1-P}



\begin{abstract}
In this paper we consider
combinatorial numbers $C_{m, k}$  for  $m\ge 1$ and $k\ge 0$ which unifies the entries of the Catalan triangles $ B_{n, k}$ and $ A_{n, k}$ for appropriate values of parameters $m$ and $k$, i.e., $B_{n, k}=C_{2n,n-k}$ and  $A_{n, k}=C_{2n+1,n+1-k}$. In fact,  some of these numbers are the  well-known Catalan numbers $C_n$  that is $C_{2n,n-1}=C_{2n+1,n}=C_n$.

 We   present  new identities for recurrence relations, linear sums  and alternating sum of  $C_{m,k}$. After that, we check  sums (and alternating sums) of squares and cubes  of $C_{m,k}$ and, consequently, for  $ B_{n, k}$ and $ A_{n, k}$. In particular, one of these equalities solves an open problem posed in \cite{[GHMN]}. We also present some linear identities involving   harmonic numbers $H_n$ and Catalan triangles numbers $C_{m,k}$. Finally, in the last section new open problems and identities
involving $C_n$ are conjectured.


 \end{abstract}

\date{}

\maketitle

\section{Introduction}
 The well-known Catalan numbers $(C_n)_{n\ge 0}$  given by the formula
$$
 C_n={1\over n+1}{2n\choose n},\quad  \ n\ge 0,
$$
appear in a wide range of problems. For instance, the Catalan number $C_n$ counts the  number of ways to triangulate a regular polygon with $n + 2$ sides; or,
the number of  ways  that   $2n$ people seat around a circular table are simultaneously
shaking hands with another person at the table in such a way that none of the arms cross each other, see for example \cite{[Sl1], [St2]}.

The Catalan numbers may be defined recursively by $C_0=1$ and
$C_n= \sum_{i=0}^{n-1} C_i C_{n-1-i}$ for $n\ge 1$ and first terms in this sequence are
$$
1,\,\,1,\,\,2,\,\,5,\,\,14, \,\,42, \,\,132, \dots
$$
Catalan numbers have been studied in depth in many papers and monographs (see for example \cite{[CC]}-\cite{[MR2]}, \cite{[Sh74]}-\cite{[St2]}) and the Catalan
sequence is probably the most frequently encountered sequence. In   \cite{[Sl1]} the generalized $k$-th Catalan numbers  $\,_k {C}_n = {1\over n}{nk\choose n-1}$, $k\ge 1$, are considered to count  the  number of ways of subdividing a convex polygon into $k$ disjoint $(n+1)$-polygons by means of non-intersecting diagonals, $k\ge 1$, see also  for example  \cite{[Chu], [HP]}.

In this paper,  we consider combinatorial numbers $(C_{m, k})_{m\ge 1,   k \ge 0 }$  given by
\begin{equation} \label{newcn}
 C_{m, k}:=\frac{m-2k}{m}{m \choose k}.                                                                                                                                                                                                                                                                                                      \end{equation}
 We collect the first values in the following table
 \begin{equation} \label{shapiro}
\begin{tabular}{c|cccccccccccc}
  $m\setminus k$&0 &1  & 2 & 3 & 4 & 5&6&7&8&9 &10&\dots
  \\
    \hline
  1 &  1 & -1 & \ &  \ &  \ &\ & \ &\ &\ &\ & \   & \   \\
   2&1 & 0 & -1& \ & \ & \  & \  & \ & \ & \ & \  & \    \\
  3& 1 & 1 &  -1& -1&\  &\  & \ & \ & \ & \ & \   & \   \\
 4& 1 & 2 & 0 & -2& -1 & \ & \ & \ & \ & \ & \   & \    \\
 5& 1 & 3 & 2 & -2 & -3 & -1 &  & \ & \ &\ & \   & \    \\
 6& 1 & 4 & 5 &  0 & -5&-4 &-1 & \   & \  & \ & \   & \     \\
  7& 1 & 5 & 9 &  5 & -5&-9 &-5 & -1   & \  & \ & \   & \     \\
   8& 1 & 6 & 14 &  14 & 0&-14 &-14 & -6   & -1  & \ & \   & \     \\
      9& 1 & 7 & 20 &  28 &  14&-14 &-28 & -20   & -7  & -1 & \     & \   \\
        10& 1 & 8 & 27 &  48 & 42&0 &-42 & -48   & -27  & -8 & -1    & \    \\
  \dots & \dots & \dots & \dots & \dots&\dots &\dots & \dots& \dots& \dots& \dots& \dots & \dots   \\
\end{tabular}
\end{equation}

These combinatorial  numbers  $(C_{m, k})_{m\ge 1,   k \ge 0 }$  are closely related to Catalan numbers $C_n$ and   generalized (or  higher) Catalan numbers $\,_k {C}_n $.  In fact,  it follows that
$$
\begin{array}{lll}
C_{2n,n-1}=C_n=C_{2n+1,n},
\vspace{0.2cm}\\
\displaystyle{C_{kn+1, n }={(k-2)n+1\over kn+1}{kn+1\choose n}=((k-2)n+1)\,_k {C}_n}.
\end{array} $$

These numbers  $(C_{m, k})_{m\ge 1,   k \ge 0 }$ appear in several Catalan triangles. For instance, $C_{2n, n-k}=B_{n, k}$, where
$$B_{n,k}={k\over n}{2n\choose n-k},\quad 0\le k\le n,
$$
(see \cite{[GHMN], [Sh74]}) and also $C_{2n+1, n+1-k}=A_{n, k}$,
where
$$
A_{n,k}={2k-1\over 2n+1}{2n+1\choose n+1-k}, \quad 1\le k\le n+1,
$$
(see \cite{[MR2]}).

 This paper is organized as follows. In the second section, we  present  a new recurrence relation   that satisfies numbers $(C_{m, k})_{m\ge 1,   k \ge 0 }$ in  Proposition \ref{recurre}. Moreover, we establish    new identities  in the sum of  $C_{m,k}$ and their alternating, $(-1)^k C_{m, k}$ in  Theorem \ref{alte}. Next, as consequence in Corollary \ref{altecubes}, we obtain    the alternating sum   of the entries
  of the two  Catalan triangle numbers $(B_{n,k})_{ n\ge  k \ge 1 }$ and $(A_{n,k})_{ n+1\ge  k \ge 1}$.

In the third section, we obtain the value of ${\sum_{k=0}^n C_{m,k}^2}$ and ${\sum_{k=0}^n (-1)^k C_{m,k}^2}$ for $m, n\ge 1$ in Theorem \ref{newsquare}. We also show two identities  which allows to decompose squares of  combinatorial  numbers as sum of squares of other combinatorial numbers. In particular, the nice  identity
$$
{2n \choose n}^2=\sum_{k=0}^{n}{3n-2k\over n}{2n-1-k\choose n-1}^2, \qquad n\ge 1,
$$
 is  presented in Theorem \ref{main}.

The forth section is dedicated to the sum of cubes of   numbers $(C_{m, k})_{m\ge 1,   k \ge 0 }$.  For $m\ge 1$ and $n\ge 1$, we present the identity
 $$
 \sum_{k=0}^{n}C_{m, k}^3=4{m-1 \choose n}^3-3{m-1 \choose n}\sum_{j=0}^{m-1}{j \choose n}{j \choose m-n-1},
 $$
in Theorem \ref{sumacubo}(i). Thus, from this identity we obtain
\begin{eqnarray*}
\sum_{k=1}^n B_{n, k}^3&=&{n+1\over 2}C_nb(n),\cr
\sum_{k=1}^{n+1} A_{n, k}^3&=&(n+1)C_n\left((2(n+1)C_n)^2-3a(n)\right),\qquad n\in \N,
\end{eqnarray*}
in  Theorem \ref{cubes2} and Corollary \ref{cubes} respectively, where integer sequences $(a(n))_{n\ge 0}$ and $(b(n))_{n\ge 1}$  are defined by $$
a(n):=\sum_{k=0}^{n}{n+k\choose n}^2\qquad \hbox{and}\qquad  b(n):=\sum_{k=0}^{n}\frac{n-k}{  n}{n-1+k\choose n-1}^2.
$$
This first sum solves the third open problem posed in \cite[Section 3]{[GHMN]}. These sequences $(a(n))_{n\ge 0}$ and $(b(n))_{n\ge 1}$  appear in the On-Line Encyclopedia of Integer Sequences  (\cite{[Sl]}). We also present  the value of the alternating sum   $\sum_{k=0}^{n}(-1)^kC_{m, k}^3$ in Theorem \ref{sumacubo}(ii).

Identities which involved harmonic numbers $(H_n)_{n\ge 1}$ where
 \begin{equation}\label{armoni}H_n=\sum_{k=1}^{n}  {1\over k}, \quad    n\in \mathbb{N},\end{equation}
 have received a notable attention in last decades. We only mention shortly papers \cite{[Chu2], [PS], [Sp]}, the monograph \cite[Chapter 7]{[BQ]} and the reference therein.

 In the fifth section we present a new identity which involves harmonic numbers $(H_n)_{n\ge 1}$ and Catalan triangle numbers $(C_{m, k})_{m\ge 1,   k \ge 0 }$ in Theorem \ref{teoharm} (and then  for $B_{n,k}$ and $A_{n,k}$ in Corollary \ref{cubes22}). This identity includes, as particular case, a known equality proved in \cite{[PS]}.

In the  last section we conjecture some  identities which involves  numbers $B_{n,k}$ and $A_{n, k}$. Although the WZ-theory (see for example \cite{[MR], [PWZ], [PS], [WZ]}) allows to give  computer proofs, authors can not find an analytic proof of these equalities. Note that  analytic proofs give additional  information about the nature of these sequences which remains hidden in computer proofs.

\medskip
\noindent{\bf Notation.} We follow the usual convention that ${u \choose v}$
is zero if $u < v$ (in particular  ${u \choose -1}$ for $u\ge 0$) and a sum is zero if its range of summation is empty.

\section{Recurrence relation and sums  of   Catalan triangle numbers}

One of  the main aim of this section is to prove a recurrence relation (Proposition \ref{recurre}) that satisfies
combinatorial numbers $(C_{m, k})_{m\ge 1,   k \ge 0 }$ given in
(\ref{newcn}). Moreover, for $m\ge 2$ and $n\ge 1$, we obtain the sum of combinatorial numbers $(C_{m, k})_{m\ge 1, k \ge 0 }$ and the alternating sum, that is,
$
\sum_{k=0}^n (-1)^kC_{m, k}
$
in Theorem \ref{alte} which includes some known identities for Catalan triangle numbers $(B_{n,k})_{ n\ge  k \ge 1 }$ and $(A_{n,k})_{ n+1\ge  k \ge 1}$.

These numbers also are related to   the entries $
B_{n,k}$ and $A_{n,k}$ of the two particular  Catalan triangles. In fact,  the combinatorial numbers $
B_{n,k}$
are the  entries of the following Catalan triangle   introduced in \cite{[Sh74]}:
\begin{equation} \label{shapiro}
\begin{tabular}{c|ccccccc}
  $n\setminus k$ &1  & 2 & 3 & 4 & 5&6&\dots
  \\
    \hline
  1 & 1 & \ &  \ &  \ &\ & \ & \\
  2 & 2 & 1& \ & \ &\  &\  & \ \\
  3 & 5 &  4& 1&\  &\  & \ & \ \\
  4 & 14 & 14 & 6& 1 &\ &\ & \ \\
  5 & 42 & 48 & 27 & 8&1 &\ & \ \\
  6 & 132 & 165 & 110 & 44&10 &1 & \  \\
  \dots & \dots & \dots & \dots & \dots&\dots &\dots & \dots \\
\end{tabular}
\end{equation}
which are given by
\begin{equation}\label{numbers1}B_{n,k}:={k\over n}{2n\choose n-k}, \ n,k\in \N,\ k\le n.
\end{equation}
Notice that $B_{n,1}=C_n$ and $C_{2n,n-k}=B_{n,k}$ for $k\le n$.

Although   numbers  $B_{n,k}$  are not as well known as
 Catalan numbers, they have also several applications, for example,
$B_{n,k}$ is the number of walks of $n$ steps, each in direction $N$, $S$, $W$ or $E$,
 starting at the origin, remaining in the upper half-plane and ending at height
 $k$;   see more details in  \cite{[DS],[Sh74],[Sl]} for more
information.

In the last years,  Catalan triangle (\ref{shapiro}) has been studied in detail.   For instance, the  formula
\begin{equation}\label{ident}
\sum_{k=1}^{i}B_{n,k}B_{n,n+k-i}(n+2k-i)=(n+1)C_n{2(n-1)\choose
i-1}, \quad i\le n,
\end{equation}
which appears in a problem related with the dynamical behavior of a
family of iterative processes has been proved in \cite[Theorem 5]{[GHMN]}.
These numbers $(B_{n,k})_{ n\ge  k \ge 1 }$ have been analyzed
in many ways. For instance,  symmetric functions have been used in \cite{[CC]},  recurrence relations  in \cite{[Sla]}, or in \cite{[GZ]}  the Newton interpolation formula, which is applied to conclude  divisibility properties of sums of products of binomial coefficients.

Other  combinatorial numbers $A_{n,k}$
defined as follows
\begin{equation}\label{numbers2}
A_{n,k}:={2k-1\over 2n+1}{2n+1\choose n+1-k}, \ n,k\in \N,\ k\le n+1,
\end{equation}
appear as the entries of this second Catalan triangle,
\begin{equation}\label{shapiro2}
\begin{tabular}{c|ccccccc}
  $n\setminus k$ &1  & 2 & 3 & 4 & 5&6&\dots
  \\
    \hline
   1 & 1 & 1 &  \ &  \ &\ & \ & \\\
  2 & 2& 3 & 1& \  &\  &\  & \ \\
  3 & 5& 9 &  5& 1&\    & \ & \ \\
  4 & 14& 28 &20 & 7& 1 &\  & \ \\
  5 & 42 & 90 & 75& 35&9&1 & \ \\
  6 & 132 & 297 & 275 & 154&54 &11 & 1  \\
  \dots & \dots & \dots & \dots & \dots&\dots &\dots & \dots \\
\end{tabular}
\end{equation}
which  is considered in \cite{[MR2]}.
Notice that $A_{n,1}=C_n$ and $C_{2n+1,n-k+1}=A_{n,k}$ for $k\le n+1$.

The entries $
B_{n,k}$ and $A_{n,k}$ of the above two particular  Catalan triangles  satisfy  the  recurrence relations
\begin{equation}\label{recurr1}B_{n,k}=B_{n-1,k-1}+2B_{n-1,k}+B_{n-1,k+1},\qquad k\ge 2,\end{equation}
and
\begin{equation}\label{recurr2}A_{n,k}=A_{n-1,k-1}+2A_{n-1,k}+A_{n-1,k+1},\qquad k\ge 2.\end{equation}

Now, we show that numbers $(C_{m, k})_{m\ge 1,    k \ge 0 }$ also satisfy a recurrence relation which extends recurrence relations (\ref{recurr1}) and (\ref{recurr2}).

\begin{proposition}\label{recurre}  For $m\ge 1$ and $ k \ge 2$, the following  identity holds:
$$
C_{m+2,k}=C_{m,k}+ 2C_{m,k-1}+C_{m,k-2}.
$$
\end{proposition}
\begin{proof} Note that
\begin{eqnarray*}
C_{m,k}+ 2C_{m,k-1}+C_{m,k-2}&=&{(m-1)!\over k!(m-k+2)!}P(m,k),
\end{eqnarray*}
where
\begin{eqnarray*}
P(m,k)&=& (m-2k)(m-k+2)(m-k+1)+2(m-2k+2)k(m-k+2)\\&\quad&\qquad +(m-2k+4)k(k-1)\\&=&m(m+1)(m+2-2k).
\end{eqnarray*}
Finally we conclude that
$$
C_{m,k}+ 2C_{m,k-1}+C_{m,k-2}={m+2-2k\over m+2}{m+2\choose k}=C_{m+2,k},
$$
and the proof is finished.
\end{proof}

As it was shown in \cite{[Sh74]}, the values of the sums of  $B_{n, k}$  and $A_{n,k}$ in terms of Catalan numbers is given by:
 \begin{equation}\label{suma1}\displaystyle{\sum_{k=1}^{n}B_{n,k}= {n+1\over
 2}C_n} \qquad   \hbox{ and }\qquad \displaystyle{\sum_{k=1}^{n+1}A_{n,k}=(n+1)C_{n}},\end{equation}
and the sums of its squares by
\begin{equation}\label{sumacuadrados}\displaystyle\sum_{k=1}^{n}B_{n,k}^2= C_{2n-1}   \qquad   \hbox{ and }\qquad \displaystyle\sum_{k=1}^{n+1}A_{n,k}^2=C_{2n}, \quad n\in \N. \end{equation}
However the sums of its cubes $ \sum_{k=1}^{n}B_{n,k}^3 $ (posed in \cite[Section 3]{[GHMN]}) and $ \sum_{k=1}^{n+1}A_{n,k}^3 $ in terms of Catalan numbers were unknown until now. This and other questions are studied in the in the next two sections.

To conclude this section we give  the sum of numbers  $(C_{m,k})_{m\ge 1,  k \ge 0 }$ and for  their alternating sum   in the following theorem.
\begin{theorem}\label{alte} For $m\ge 2$ and $n \ge 1$, we obtain the following identities:
\begin{itemize}
\item[(i)]${\displaystyle\sum_{k=0}^n C_{m,k}= {m-1\choose n}}$,

\item[(ii)]
${\displaystyle\sum_{k=0}^n (-1)^kC_{m,k}=(-1)^nC_{m-1,n}}$.
\end{itemize}
\end{theorem}

\begin{proof} Note that it is enough to check by induction process the  identities. We only prove  item (ii).
For $n=1$, we directly check  the identity. Now suppose that the identity holds for $n$. Then
\begin{eqnarray*}
\sum_{k=0}^{n+1}(-1)^k C_{m,k}&=&(-1)^nC_{m-1,n}+(-1)^{n+1}C_{m,n+1}\\
 &=&(-1)^n{m-2n-1\over m-1}{m-1\choose n}+(-1)^{n+1}{m-2n-2\over  n+1}{m-1\choose n}\\
 &=&(-1)^{n+1} {(m-2n-3)(m-n-1)\over (m-1)(n+1)}{m-1\choose n} \cr &=& (-1)^{n+1}{m-2n-3\over m-1}{m-1\choose n+1}=(-1)^{n+1}C_{m-1,n+1}.
\end{eqnarray*}
\end{proof}
Notice that   item (i) in  Theorem \ref{alte} includes the identities given in (\ref{suma1}). On the other hand, item (i) in the next corollary  was proved in \cite{[Ep]} and we present an alternative proof.
\begin{corollary} \label{altecubes}
For $n\ge 1$, we have
\begin{itemize}
 \item[(i)]$\displaystyle{\sum_{k=1}^{n}(-1)^{k} B_{n,k}=-C_{n-1}}$,
\item[(ii)] $\displaystyle{\sum_{k=1}^{n+1}(-1)^k A_{n,k}=0}$.
\end{itemize}
\end{corollary}

\begin{proof} By Theorem \ref{alte}, we have
$$
\sum_{k=1}^{n}(-1)^{k} B_{n,k}= \sum_{k=0}^{n}(-1)^{k} C_{2n,n-k}=\sum_{k=0}^{n}(-1)^{n-k} C_{2n,k}= C_{2n-1,n}=-C_{n-1},$$
and
$$
\sum_{k=1}^{n+1}(-1)^k A_{n,k}= \sum_{k=1}^{n+1}(-1)^{k} C_{2n+1,n-k+1}=\sum_{k=0}^{n}(-1)^{n-k+1} C_{2n+1,k}=-C_{2n,n}=0.
$$
\end{proof}

\section{Sums of  squares  of Catalan triangle numbers  }
\setcounter{theorem}{0}
\setcounter{equation}{0}
In the section, our main
objective is twofold. Firstly, we check $
\sum_{k=0}^n C^2_{m,k} $ and $
\sum_{k=0}^n (-1)^k C^2_{m,k} $ in Theorem \ref{newsquare}.  As a consequence of this result, the  identities presented in (\ref{sumacuadrados}) are proved in Corollary \ref {otros}.

Secondly, a key result of this paper is to decompose the binomial number $ {2n \choose n}^2$ in  sum of squares of other combinatorial numbers, i.e.
$$
{2n \choose n}^2=\sum_{k=0}^{n}{3n-2k\over n}{2n-1-k\choose n-1}^2, \qquad n\ge 1.
$$
To do that we present a straightforward proof as a consequence of a more general  identity in combinatorial numbers in Theorem \ref{main} (i). This equality is essential to check ${\sum_{k=1}^{n}B_{n,k}^3}$  in Theorem \ref{cubes2}.

\begin{theorem}\label{newsquare}For $n\ge 1$ and $m\ge 1$, we have
\begin{itemize}\item[(i)] $\displaystyle{
\sum_{k=0}^n C^2_{m,k}= {m-2n\over m}{m-1\choose n}^2+{2\over m}\sum_{k=0}^{n-1}{m-1\choose k}^2,}$
\item[(ii)] $\displaystyle{
\sum_{k=0}^n (-1)^kC^2_{m,k}= 2(-1)^n{m-1\choose n}^2-\sum_{k=0}^{n}(-1)^k{m\choose k}^2.}
$
\end{itemize}
\end{theorem}

\begin{proof} We prove the identities  by invoking  an inductive process for $n$.

(i)  For $n=1$, we directly check  it.
Now we assume that the desired identity holds for $n$. For   $n+1$, we have
\begin{eqnarray*}
\sum_{k=0}^{n+1} C^2_{m,k}&=& {m-2n\over m}{m-1\choose n}^2+{2\over m}\sum_{k=0}^{n-1}{m-1\choose k}^2+\left({m-2n-2\over m}{m\choose n+1}\right)^2.
\end{eqnarray*}
On the other hand, observe that $$
\left({m-2n-2\over m}{m\choose n+1}\right)^2={m-2n-2\over m}{m-1\choose n+1}^2-{m-2n\over m}{m-1\choose n}^2+{2\over m}{m-1\choose n}^2.
$$Therefore, we obtain the identity
$$
\sum_{k=0}^{n+1} C^2_{m,k}= {m-2(n+1)\over m}{m-1\choose n+1}^2+{2\over m}\sum_{k=0}^{n}{m-1\choose k}^2.
$$

(ii) For $n=1$, we directly check it. Now we assume that the desired identity holds for $n$. For $n+1$, we have
$$
\sum_{k=0}^{n+1}(-1)^k C^2_{m,k}= 2(-1)^n{m-1\choose n}^2-\sum_{k=0}^{n}(-1)^k{m\choose k}^2
+(-1)^{n+1}\left({m-2n-2\over m}{m\choose n+1}\right)^2.
$$
On the other hand, observe that
$$
\left({m-2n-2\over m}{m\choose n+1}\right)^2=2{m-1\choose n+1}^2+ 2{m-1\choose n}^2-{m\choose n+1}^2.
$$
Therefore, we obtain  the identity
$$
\sum_{k=0}^{n+1} (-1)^kC^2_{m,k}= 2(-1)^{n+1}{m-1\choose n+1}^2-\sum_{k=0}^{n+1}(-1)^k{m\choose k}^2.
$$
\end{proof}

Now, taking into account the well-known Vandermonde identity
$ \sum_{k=0}^{n}{n\choose k}^2={2n\choose n} $ and identity $ \sum_{k=0}^{2n}(-1)^k{2n\choose k}^2=(-1)^n{2n\choose n} $ for $n\ge 0$, the following corollary is obtained.  Note that the Corollary \ref{otros} (iv) was proved in \cite[Theorem 2.2]{[Zhang]}.

\begin{corollary}\label{otros} For $n\ge  1$, we have
\begin{itemize}
\item[(i)] $\displaystyle{\sum_{k=0}^n C^2_{n,k}= 2 C_{n-1},}$
\item[(ii)] $\displaystyle{\sum_{k=1}^n B^2_{n,k}=  C_{2n-1},}$
\item[(iii)]$\displaystyle{\sum_{k=1}^{n+1} A^2_{n,k}=  C_{2n},}$
\item[(iv)] $\displaystyle{\sum_{k=1}^n (-1)^kB^2_{n,k}= -{n+1\over 2} C_{n}.}$
\end{itemize}
\begin{proof}
 From Theorem \ref{newsquare} (i), we  have
$$
\sum_{k=1}^n C^2_{n,k}={2\over n}\sum_{k=0}^{n-1}{n-1\choose k}^2=2C_{n-1},
$$
$$
\sum_{k=0}^n B^2_{n,k} =  \sum_{k=0}^n C^2_{2n,k}={2\over 2n}\sum_{k=0}^{n-1}{2n-1\choose k}^2={1\over 2n}\sum_{k=0}^{2n-1}{2n-1\choose k}^2= C_{2n-1},
$$
and
$$
\sum_{k=1}^{n+1} A^2_{n,k} =  \sum_{k=0}^{n} C^2_{2n+1,k}={1\over 2n+1}\left({2n\choose n}^2+ 2 \sum_{k=0}^{n-1}{2n\choose k}^2\right)={1\over 2n+1} \sum_{k=0}^{2n}{2n\choose k}^2 =C_{2n}.
$$
 As a consequence of Theorem \ref{newsquare} (ii), item (iv) is obtained
 \begin{eqnarray*}
\sum_{k=0}^n (-1)^k B^2_{n,k}&=& \sum_{k=0}^n (-1)^{n+k}  C^2_{2n,k}=2{2n-1\choose n}^2-\sum_{k=0}^{n }   (-1)^{n+k}     {2n\choose k}^2\cr&=&{-1\over 2 }   \sum_{k=0}^{2n } (-1)^{n+k}    {2n \choose k}^2= {-1\over 2 }  {2n \choose n}=-\frac{n+1}{2} C_{n}.
\end{eqnarray*}
\end{proof}
\end{corollary}

\begin{theorem}\label{main}For $m\ge n\ge 1$, we have
\begin{itemize}\item[(i)]$ \displaystyle{{m \choose n}^2=\sum_{j=n}^{m}\frac{2j-n}{n}{j-1 \choose n-1}^2,}$
\item[(ii)] $\displaystyle{
{2n \choose n}^2=\sum_{k=0}^{n}{3n-2k\over n}{2n-1-k\choose n-1}^2.}
$
\end{itemize}
\end{theorem}

\begin{proof} To prove item (i) we   invoke  an inductive process for  $m$.  For $m=n$, we check directly the identity. Now we assume that the identity  holds for $m$ and we prove it for $m+1$. Thus, it follows that
\begin{eqnarray*}
 &\quad&\sum_{j=n}^{m+1}\frac{2j-n}{n}{j-1 \choose n-1}^2= {m \choose n}^2+\frac{2m+2-n}{n}{m \choose n-1}^2\cr
 &\quad&\qquad= \left(\frac{m-n+1}{m+1}{m+1 \choose n}\right)^2+\frac{2m+2-n}{n}\left(\frac{n}{m+1}{m+1 \choose n}\right)^2= {m+1 \choose n}^2.
 \end{eqnarray*}
Then we conclude the identity holds for $m\ge n\ge 1$.

 To show item (ii) observe that
\begin{eqnarray*} \sum_{k=0}^{n}\frac{2(2n-k)-n}{n} {2n-k-1 \choose n-1}^2= \sum_{j=n}^{2n}\frac{2j-n}{n}{j-1 \choose n-1}^2={2n \choose n}^2,
\end{eqnarray*}
where we have applied item (i).
\end{proof}

\noindent{\bf Remark.} Notice that item (ii) in Theorem \ref{main} gives a decomposition of sum of squares of $ {2n \choose n}^2$ for $n\ge 1$, that can be written in this form
$$
{2n \choose n}^2= \sum_{j=0}^{n}{n+2j\over n}{n-1+j\choose n-1}^2.
$$

\section{Sums of  cubes  of Catalan triangle numbers }
\setcounter{theorem}{0}
\setcounter{equation}{0}

In this section  we  check  the sum of cubes and alternating cubes of numbers $(C_{m, k})_{m\ge 1,    k \ge 0 }$ in Theorem \ref{sumacubo}. For $m\ge 1$ and $n\ge 1$, we use  the  identity
\begin{equation}\label{AMM}
 \sum_{k=0}^{n}(m-2k){m \choose k}^3=(m-n){m \choose n}\sum_{j=0}^{m-1}{j \choose n}{j \choose m-n-1},
 \end{equation}
which is proven in \cite{[OT]}. We also present some expressions of ${\sum_{k=1}^{n}B_{n,k}^3}$, ${\sum_{k=1}^{n+1}A_{n,k}^3}$ and ${\sum_{k=1}^{n+1}(-1)^kA_{n,k}^3}$ in Corollary \ref{cubes}. The equality presented in Theorem \ref{cubes2}
 solves the  third open problem posed in \cite[Section 3]{[GHMN]}.

\begin{theorem}\label{sumacubo} For $m\ge 1$ and $n\ge 1$,we have
\begin{itemize}
 \item[(i)] $\displaystyle{\sum_{k=0}^{n}C_{m,k}^3=4{m-1 \choose n}^3-3{m-1 \choose n}\sum_{j=0}^{m-1}{j \choose n}{j \choose m-n-1},}$
 \item[(ii)] $\displaystyle{\sum_{k=0}^n(-1)^k C^3_{m,k}={m-3n\over m}(-1)^n{m-1\choose n}^3-{m-3\over m}\sum_{k=0}^{n-1}(-1)^k{m-1\choose k}^3}.
$
 \end{itemize}
\end{theorem}

\begin{proof}
(i) We apply (\ref{AMM}) to get that
\begin{eqnarray}
&\,& m^3\sum_{k=0}^{n}C_{m,k}^3+3m^3{m-1\choose n}\sum_{j=0}^{m-1}{j\choose n}{j\choose m-n-1} \nonumber \cr
&\,&  =m^3\sum_{k=0}^{n}C_{m,k}^3+3m^2\sum_{k=0}^{n}(m-2k){m\choose k}^3\nonumber\cr
&\,&  =\sum_{k=0}^{n}\left((m-2k)^3+3m^2(m-2k)\right){m\choose k}^3\nonumber\cr
&\,&=\sum_{k=0}^{n}4\left((m-k)^3-k^3\right){m \choose k}^3=4\sum_{k=0}^{n}\left((m-k)^3{m \choose m-k}^3-k^3{m \choose k}^3\right)\nonumber\cr
&\,& =4\sum_{k=0}^{n}\left(m^3{m-1 \choose m-k-1}^3-m^3{m-1 \choose k-1}^3\right)\nonumber\cr
&\,& =4m^3\sum_{k=0}^{n}\left({m-1 \choose k}^3-{m-1 \choose k-1}^3\right)\nonumber= 4m^3\left({m-1 \choose n}^3-{m-1 \choose -1}^3\right)\cr
&\,& =4m^3{m-1 \choose n}^3.
\end{eqnarray}
Therefore, we obtain the desired identity.

(ii) We prove the identity by invoking an inductive process for $n$. For $n=1$, we directly check it. Now we assume that the desired identity holds for $n$. For $n+1$, we have
\begin{eqnarray*}\sum_{k=0}^{n+1}(-1)^k C^3_{m,k}&=&{m-3n\over m}(-1)^n{m-1\choose n}^3-{m-3\over m}\sum_{k=0}^{n-1}(-1)^k{m-1\choose k}^3\cr&\quad&\qquad \qquad+(-1)^{n+1}\left({m-2n-2\over m}{m\choose n+1}\right)^3.
\end{eqnarray*}
On the other hand, observe that
$$
\left({m-2n-2\over m}{m\choose n+1}\right)^3={m-3n-3\over m}{m-1\choose n+1}^3+ {2m-3n-3\over m}{m-1\choose n}^3.
$$
Therefore, we obtain  the identity
$$\sum_{k=0}^{n+1}(-1)^k C^3_{m,k}={m-3(n+1)\over m}(-1)^{n+1}{m-1\choose n+1}^3-{m-3\over m}\sum_{k=0}^{n}(-1)^k{m-1\choose k}^3.
$$

\end{proof}

As a nice consequence of Theorem \ref{sumacubo}, we obtain expressions of $\sum_{k=1}^{n}B_{n,k}^3$,  $\sum_{k=1}^{n+1}A_{n,k}^3$ and $ \sum_{k=1}^{n+1}(-1)^kA_{n,k}^3$ in the next corollary. To check this last sum, we use  Dixon's identity, $$\sum_{k=0}^{2n}(-1)^k{2n\choose k}^3=(-1)^n{2n \choose n}{3n\choose n}, \qquad n\ge 1.$$

\begin{corollary}\label{cubes} For $n\ge 1$, we have
\begin{itemize}
 \item[(i)]$\displaystyle \sum_{k=0}^{n}B_{n,k}^3=\frac{1}{2}{2n \choose n}^3-\frac{3}{2}{2n \choose n}\sum_{j=n}^{2n-1}{j \choose n}{j \choose n-1}$,
\item[(ii)] $\displaystyle \sum_{k=1}^{n+1}A_{n,k}^3={2n \choose n}^3-3{2n \choose n}\sum_{j=n}^{2n-1}{j \choose n}^2$,
 \item[(iii)] $\displaystyle{
\sum_{k=1}^{n+1}(-1)^k A_{n,k}^3={n-1\over 2n+1}{2n \choose n}{3n\choose n}.}$
\end{itemize}
\end{corollary}
\begin{proof}  From  Theorem \ref{sumacubo}(i), we have
 \begin{eqnarray*}
 \sum_{k=0}^{n}B_{n,k}^3&=&\sum_{k=0}^{n}C_{2n,k}^3=4{2n-1 \choose n}^3-3{2n-1 \choose n}\sum_{j=0}^{2n-1}{j \choose n}{j \choose n-1}\cr
 &=&\frac{1}{2}{2n \choose n}^3-\frac{3}{2}{2n \choose n}\sum_{j=n}^{2n-1}{j \choose n}{j \choose n-1},\cr
 \end{eqnarray*}
and
$$ \sum_{k=1}^{n+1}A_{n,k}^3=\sum_{k=0}^{n}C_{2n+1,k}^3\displaystyle =4{2n \choose n}^3-3{2n \choose n}\sum_{j=0}^{2n}{j \choose n}{j \choose n}={2n \choose n}^3-3{2n \choose n}\sum_{j=n}^{2n-1}{j \choose n}^2.$$

Now, from  Theorem \ref{sumacubo}(ii), we have
 \begin{eqnarray*}
 \sum_{k=1}^{n+1}(-1)^kA_{n,k}^3&=&\sum_{k=0}^{n}(-1)^{n+1-k}C_{2n+1,k}^3 ={n-1\over 2n+1}\left({2n \choose n}^3+2\sum_{k=0}^{n-1}(-1)^{n+k}{2n\choose k}^3\right)\cr
  &=&{n-1\over 2n+1}   \sum_{k=0}^{2n}(-1)^{n+k}{2n\choose k}^3 ={n-1\over 2n+1}{2n \choose n}{3n\choose n}.\cr
 \end{eqnarray*}
\end{proof}

\noindent{\bf Remark.} Note that the part (ii) of Corollary \ref{cubes}  may be written as
$$
\sum_{k=1}^{n+1} A_{n, k}^3=(n+1)C_n\left((2(n+1)C_n)^2-3a(n)\right),\qquad n\ge 1,
$$
where the integer sequence of numbers $(a(n))_{n\ge 0}$ is defined by
$$
a(n):=\sum_{k=0}^n{n+k\choose n}^2, \qquad n \in \N\cup\{0\}.
$$
Note that $a(0)=1,$ $a(1)=5$, $a(2)= 46$, $a(3)= 517$, $a(4)=6376$... . This sequence appears indexed in the  On-Line Encyclopedia of Integer Sequences by N.J.A. Sloane  (\cite{[Sl]}) with the reference $A112029$.

 The sequence $\left({2n\choose n}{3n\choose n}\right)_{n\ge 0}$ is known as De Bruijn's $S(3,n)$ and appears in the Sloane's On-Line Encyclopedia  with the reference $A006480$.

\begin{theorem}\label{cubes2} For $n \ge 1$,  the following identity holds:
$$\sum_{k=1}^{n}B_{n,k}^3
={1\over 2n}{2n \choose n}\sum_{k=1}^nk{2n-k-1\choose n-1}^2.
$$

\end{theorem}
\begin{proof}
By Corollary \ref{cubes} (i), we have
\begin{eqnarray*}\sum_{k=1}^{n}B_{n,k}^3
=\frac{1}{2}{2n \choose n}\left({2n \choose n}^2-3\sum_{j=n}^{2n-1}{j \choose n}{j \choose n-1}\right)\cr
\end{eqnarray*}
and we claim  that
$$\displaystyle {2n \choose n}^2-3\sum_{j=n}^{2n-1}{j \choose n}{j \choose n-1}=\sum_{k=1}^{n}\frac{k}{n}{2n-k-1 \choose n-1}^2.$$
Note that
$$
 \sum_{j=n}^{2n-1}{j \choose n}{j \choose n-1}=\sum_{j=n-1}^{2n-1}\frac{j-n+1}{n}{j \choose n-1}^2=\sum_{k=0}^{n}\frac{n-k}{n}{2n-k-1 \choose n-1}^2,$$
and then we obtain
$$
3\sum_{j=n}^{2n-1}{j \choose n}{j \choose n-1}+\sum_{k=0}^{n}\frac{k}{n}{2n-k-1 \choose n-1}^2=\sum_{k=0}^{n}\frac{3n-2k}{n}{2n-k-1 \choose n-1}^2={2n \choose n}^2,$$
where we have applied  Theorem \ref{main} (ii).
\end{proof}

\noindent{\bf Remark.} Note that the  identity of Theorem \ref{cubes2} may be written as
$$
\sum_{k=1}^n B_{n, k}^3={n+1\over 2}C_nb(n),
$$
where the integer sequence of numbers $(b(n))_{n\ge 1}$ is defined by
$$
b(n):=\sum_{k=0}^n{k\over n}{2n-k-1\choose n-1}^2=\sum_{k=0}^n{n-k\over n}{n-1+k\choose n-1}^2 , \qquad n \in \N.
$$
Note that $b(1)=1,$ $b(2)=3$, $b(3)= 19$, $b(4)= 163$, $b(5)=1625,\ldots$. This sequence also appears indexed in the  On-Line Encyclopedia of Integer Sequences by N.J.A. Sloane (\cite{[Sl]}) with the reference $A183069$.

\section{Identities involving   harmonic numbers  and  Catalan triangle numbers }

\setcounter{theorem}{0}
\setcounter{equation}{0}

A large number of identities which included harmonic numbers $(H_n)_{n\ge 1}$, defined by (\ref{armoni}), have appeared in several papers: a systematic study of explicit formulas for sums of the form $\sum_{k=1}^n a_k H_k$ are given in  \cite{[Sp]}; some other finite summation identities
involving harmonic numbers are considered in \cite{[PS]} and proved by the WZ-theory; infinite series involving harmonic numbers are presented in
 \cite{[Chu2]}. See other approaches in \cite[Chapter 7]{[BQ]} and reference therein.

 However, the next nice relation between  Catalan triangle numbers $(C_{m, k})_{m\ge 1,    k \ge 0 }$ and
harmonic numbers $(H_n)_{n\ge 1}$ seems to be new. We also present the particular case  of $B_{n,k}$ and $A_{n,k}$ in Corollary \ref{cubes22}.

\begin{theorem}\label{teoharm} For $m\ge 1$ and $n\ge 1$, we have
\begin{equation}\label{harm}
\sum_{k=1}^{n}C_{m,k} H_k
={m-1 \choose n}H_n - {1\over m}\sum_{k=1}^n {m\choose k}.
\end{equation}

\end{theorem}
\begin{proof} We prove the identities by invoking   an induction process for $n$. For $n=1$, we directly check it.  Now we assume that the identity  (\ref{harm}) holds for $n$. For $n+1$, we have
$$\sum_{k=1}^{n+1}C_{m,k} H_k={m-1 \choose n}H_n - {1\over m}\sum_{k=1}^n{m\choose k} +{m-2n-2\over m}{m \choose n+ 1}   H_{n+1}.
$$
On the other hand, observe that
$$ {m-1 \choose n}H_n - { n+1\over m}{m \choose n+ 1}   H_{n+1}=-{1\over m} {m\choose n+1}.
$$
Therefore, we obtain the identity
$$
\sum_{k=1}^{n+1}C_{m,k} H_k
={m-n-1\over m}{m \choose n+1}H_{n+1} - {1\over m}\sum_{k=1}^{n+1} {m\choose k}={m-1 \choose n+1}H_{n+1}- {1\over m}\sum_{k=1}^{n+1} {m\choose k}.
$$
\end{proof}

Using Theorem \ref{teoharm}, we will show the relationship of the harmonic numbers and the Catalan triangle numbers.

\begin{corollary}\label{cubes22} For $n \ge 1,$ we have
\begin{itemize}
\item[(i)] $
\displaystyle{\sum_{k=1}^{n}C_{n,k} H_{k}
 = {1-2^n\over n},}$
\item[(ii)]$
\displaystyle{\sum_{k=0}^{n-1}B_{n,k} H_{n-k}}
 =  \displaystyle{{2nH_n-1\over 4n}{2n \choose n} - {2^{2n-1}-1\over 2n},}$
\item[(iii)]$
\displaystyle{\sum_{k=1}^{n}A_{n,k} H_{n-k+1}}
 =  \displaystyle{H_n{2n \choose n} - {2^{2n}-1\over 2n+1}}.$
\end{itemize}

\end{corollary}
\begin{proof}   On the one hand, we have
$$
\displaystyle{\sum_{k=1}^{n}C_{n,k} H_{k}={-1\over n}\sum_{k=1}^n {n\choose k}
 = {1-2^n\over n}.}
 $$
On the other hand, taking into account   identities $$\displaystyle{2\sum_{k=0}^{n}{2n\choose k}-{2n\choose n}=\sum_{k=0}^{2n}{2n\choose k}=2^{2n}}\quad \mbox{and} \quad  \displaystyle{2\sum_{k=0}^{n}{2n+1\choose k}=\sum_{k=0}^{2n+1}{2n+1\choose k} =2^{2n+1} },$$
we have
$$
\displaystyle{\sum_{k=0}^{n-1}B_{n,k} H_{n-k}}=\sum_{k=1}^nC_{2n, k}H_k=H_n{2n-1\choose n}-{1\over 2n}\sum_{k=1}^n{2n \choose k}={2nH_n-1\over 4n}{2n \choose n} - {2^{2n-1}-1\over 2n},
$$
and
$$
\displaystyle{\sum_{k=1}^{n}A_{n,k} H_{n-k+1}}=\sum_{k=1}^nC_{2n+1, k}H_k=H_n{2n\choose n}-{1\over 2n+1}\sum_{k=1}^n{2n+1 \choose k}={H_n}{2n \choose n} - {2^{2n}-1\over 2n+1}.
$$

\end{proof}
\noindent{\rm {\bf{Remark.}} By Corollary \ref{cubes22} (i), we have
$$
\sum_{k=1}^{n}(n-2k) H_k{n \choose k}
={1-2^n},
$$
which was shown in  \cite[Formula (13)]{[PS]}.
}

\section{New conjectures, final comments and conclusions}
\setcounter{theorem}{0}
\setcounter{equation}{0}

In this last section, we present two conjectures about new identities in Catalan triangle numbers. We have directly checked that these identities  hold for first values of $n$ and $m$. Although  analytic proofs are not yet available,  alternative proofs as to apply WZ-theory (\cite{[PS], [WZ]}) or some mathematical software, indicate us that these equalities hold. Note that an analytic proof will give us some extra information about these nature of the sums. To conclude the paper, we present some final comments and conclusions.

\begin{conjecture}\label{firstcon}{\rm For $m>n \ge 1$ and an odd integer  $p$, the factor $ m-1\choose n $ divides
$ \sum_{k=0}^nC_{m,k}^p.$ Note that  the conjecture holds for  $p=1$ and $p=3$, see Theorem \ref{alte} (i) and Theorem \ref{sumacubo} respectively. Now we present two important cases of this conjecture.
\begin{itemize}

\item[(i)] Taking into count that  $B_{n, k}=C_{2n,n-k}$, a positive answer of the conjecture \ref{firstcon} would imply that the factor ${n+1\over 2}C_n$ divides
${\sum_{k=1}^nB_{n,k}^p}
$ for $n \ge 1 $ and an odd integer  $p$. In the case that $p=1$ and $p=3$ the sums are explicitly given in (\ref{suma1}) and Corollary \ref{cubes}(i) respectively. We have directly checked that the factor ${n+1\over 2}C_n$ divides
${\sum_{k=1}^nB_{n,k}^5}
$ for first values of $n$.

\item[(ii)] Now we considerer that $A_{n, k}=C_{2n+1,n+1-k}$. A positive answer of the conjecture \ref{firstcon} would imply that the factor $(n+1)C_n$ divides
$\sum_{k=1}^{n+1}A_{n,k}^p$ for $n \ge 1 $ and an odd integer  $p$. For $p=1$ and $p=3$ these sums are given in (\ref{suma1}) and Corollary \ref{cubes} (ii) respectively. We have also checked that ${(n+1)}C_n$ divide
${\sum_{k=1}^{n+1}A_{n,k}^5}
$ for first values of $n$.

\end{itemize}

}\end{conjecture}

\bigskip
\begin{conjecture} {\rm For $n,m\in \N$, the identity
$$
\displaystyle \sum_{k=1}^{r}B_{n,k}^2B_{m,k}=\frac{1}{2}{2n \choose n}^2{2m \choose m}\biggl[1-\frac{n+2m}{r}{n+m \choose n}^{-1}{n+r \choose n}^{-1}\sum_{j=0}^{r-1}{s+j \choose s}{n+j \choose n-1}\biggr],
$$
holds where $r=\min(n,m)$ and $s=\max(n,m)$. In the particular case, $m=n$, we recover the  identity given in Corollary \ref{cubes} (i). Note that the nature of this formula is different than the formula (\ref{ident}).
}
\end{conjecture}

\subsection*{Final comments and conclusions} In this paper we have presented a unified study of two families of Catalan triangle numbers. We have considered finite sums of powers (linear, squares and cubes) of  these numbers to show original (and nice) identities involving Catalan numbers (section 2-4). Some of these equalities solve some open problems and connect Catalan sequences with other some known sequences, see for example Theorem \ref{cubes2}.  Note that we have not considered moments on these sums of powers as in other  papers in the literature, see for example \cite{[CC], [MR2], [Zhang]}. We have also presented a natural connection between harmonic numbers and Catalan triangle numbers which seems to be new and  may be completed in later studies. Finally some conjectures about other sums of Catalan triangle numbers are posed.



\begin{thebibliography}{999}

\bibitem{[BQ]}
A.T. Benjamin and J. Quinn: {\it Proofs that Really Count,}  Dolciani Mathematical Exposition \textbf{27}, Mathematical Association of
America, 2003.

\bibitem{[Chu]}  W. Chu: {\it A new combinatorial interpretation for generalized Catalan numbers}. { Discrete Math.},  {\bf 65}
(1987), 91--94.

\bibitem{[CC]}  X. Chen and  W. Chu: {\it Moments on  Catalan number}. { J. Math. Anal. Appl.},  {\bf 349}
(2009), no. 2, 311--316.


\bibitem{[Chu2]}  W. Chu: {\it Summation formulae involving harmonic numbers}. { Filomat},  {\bf 26}
(2012), no. 1, 143--152.




\bibitem{[DS]} E. Deutsch and L. Shapiro: {\it A survey of the Fine numbers,}  Discrete Math.,
  {\bf 241} (2001), 241--265.



\bibitem{[Ep]} W.J.R. Eplett: {\it A note about the Catalan triangle,}  Discrete Math., {\bf 25} (1979), 289--291.


\bibitem{[GZ]} V.J.W. Guo and J. Zeng: {\it Factors of binomial sums from Catalan triangle}. { J. Number Theory},  {\bf 130}
(2010), no. 1, 172--186.

\bibitem{[GHMN]}  J. M. Guti\'{e}rrez, M.A. Hern\'{a}ndez, P.J. Miana, and N. Romero: {\it New identities in the Catalan triangle}. { J. Math. Anal. Appl.},  {\bf 341}
(2008), no. 1, 52--61.

 \bibitem{[HP]} P. Hilton and J. Pedersen: {\it   Catalan numbers, their generalization and their uses,} Math. Intelligencer {\bf 13} (1991), 64--75.

\bibitem{[MR]} P.J. Miana and N. Romero: {\it Computer proofs of new identities in the Catalan triangle}.  {Biblioteca de la Revista Matem\'{a}tica Iberoamericana.
Proc. of the ``Segundas Jornadas de Teor\'{\i}a de N\'{u}meros'', (Madrid, 2007), 203–-208}.


\bibitem{[MR2]} P.J. Miana and N. Romero: {\it   Moments of combinatorial and Catalan numbers,} J. Number Theory,  {\bf 130} (2010), no. 8, 1876--1887.



\bibitem{[OT]} H. Ohtsuka and R. Tauraso, \textit{Problem 11844}, Amer. Math. Monthly 122.5(2015):501.

    Solution http://www.mat.uniroma2.it/$\sim$tauraso/AMM/AMM11844.pdf



\bibitem{[PS]}P. Paule   and C. Schneider. {\it Computer proofs of a new family of harmonic
number identities,}  	Adv. Appl. Math.,  {\bf 31} (2003), no. 2,  359–-378.




\bibitem{[PWZ]} M. Petkovsek, H. S. Wilf and D. Zeilberger:
$A=B$. A. K. Peters Ltd., Wellesley, 1997.
http://www.cis.upenn.edu/$\sim$wilf/AeqB.html






\bibitem{[Sh74]} L.~W. Shapiro: {\it A  Catalan triangle,} Discrete Math.,
  {\bf 14} (1976), 83--90.


\bibitem{[Sla]} A. Slav\'{\i}k: {\it Identities with squares of binomial coefficients,} Ars Combinatoria, {\bf 113} (2014), 377--383.


\bibitem{[Sl1]} N. Sloane, A Handbook of  Integer Sequences, Academic Press, 1973.

\bibitem{[Sl]} N. Sloane, The On-line Encyclopedia of Integer Sequences(OEIS) http://oeis.org/.

\bibitem{[Sp]} J. Spies: {\it Some identities involving harmonic numbers,} Math. of Computation {\bf 55} (1990), 839--863.





\bibitem{[St]}
R.~P. Stanley: {\it Enumerative Combinatorics,} vol. 2, Cambridge
University Press, 1999.


\bibitem{[St2]}
R.~P. Stanley: {\it Catalan Numbers,} Cambridge
University Press, 2015.

\bibitem{[WZ]} H. Wilf and D. Zeilberger: {\it Rational functions certify combinatorial,}
J. Amer. Math. Soc., {\bf 3} (1990), no. 1,  147--158.

\bibitem{[Zhang]} Z. Zhang and B. Pang: {\it Several identities in the Catalan triangle,}
 Indian J. Pure Appl. Math.,  {\bf 41} (2010), no. 2,  363--378.



\end{thebibliography}
\end{document}